\documentclass[final,3p]{elsarticle}

\usepackage{amssymb}
\usepackage{amsmath}
\usepackage{amsthm}
\usepackage{url}

\usepackage{algorithm}
\usepackage{algorithmic}
\usepackage{complexity}

\newtheorem{theorem}{Theorem}
\newtheorem{definition}[theorem]{Definition}
\newtheorem{lemma}[theorem]{Lemma}

\newtheorem{conjecture}[theorem]{Conjecture}
\newtheorem{problem}[theorem]{Problem}

\journal{Linear Algebra and its Applications}

\begin{document}

\begin{frontmatter}
\title{Phase retrieval from very few measurements}

\author[afit]{Matthew Fickus}
\author[afit]{Dustin G.~Mixon}
\ead{dustin.mixon@afit.edu}
\author[afit]{Aaron A.~Nelson}
\author[msu]{Yang Wang}

\address[afit]{Department of Mathematics and Statistics, Air Force Institute of Technology, Wright-Patterson AFB, OH 45433, USA}
\address[msu]{Department of Mathematics, Michigan State University, East Lansing, MI 48824, USA}

\begin{abstract}
In many applications, signals are measured according to a linear process, but the phases of these measurements are often unreliable or not available.
To reconstruct the signal, one must perform a process known as phase retrieval.
This paper focuses on completely determining signals with as few intensity measurements as possible, and on efficient phase retrieval algorithms from such measurements.
For the case of complex $M$-dimensional signals, we construct a measurement ensemble of size $4M-4$ which yields injective intensity measurements; this is conjectured to be the smallest such ensemble.
For the case of real signals, we devise a theory of ``almost'' injective intensity measurements, and we characterize such ensembles.
Later, we show that phase retrieval from $M+1$ almost injective intensity measurements is $\NP$-hard, indicating that computationally efficient phase retrieval must come at the price of measurement redundancy.
\end{abstract}


\end{frontmatter}

\section{Introduction}

Given an ensemble $\Phi=\{\varphi_n\}_{n=1}^N$ of $M$-dimensional vectors (real or complex), the \textit{phase retrieval problem} is to recover a signal $x$ from \textit{intensity measurements} $\mathcal{A}(x):=\{|\langle x,\varphi_n\rangle|^2\}_{n=1}^N$. 
Note that for any scalar $\omega$ of unit modulus, $\mathcal{A}(\omega x)=\mathcal{A}(x)$, and so the best one can hope to do is recover $x$ up to a global phase factor $\{\omega x:|\omega|=1\}$. 
Intensity measurements arise in a number of applications in which phase is either unreliable or not available~\cite{BunkEtal:07,DaintyF:87,Harrison:93,MiaoISE:08,Millane:90,Walther:63}, and in most of these applications, it is desirable to perform phase retrieval from as few measurements as possible; indeed, increasing $N$ invariably makes the measurement process more expensive or time consuming. 

Recently, there has been a lot of work on algorithmic phase retrieval.
For example, phase retrieval can be formulated as a low-rank (actually, rank-1) matrix recovery problem~\cite{CandesESV:13,CandesL:12,CandesSV:13,ChaiMP:11,DemanetH:12,Voroninski:12a}, and with this formulation, phase retrieval is possible from $N=O(M)$ intensity measurements~\cite{CandesL:12}.
Another approach is to exploit the polarization identity along with expander graphs to design a measurement ensemble and apply spectral methods to perform phase retrieval~\cite{AlexeevBFM:12,BandeiraCM:13}.
One can also formulate phase retrieval in terms of MaxCut, and solvers for this formulation are equivalent to a popular solver (PhaseLift) for the matrix recovery formulation~\cite{Voroninski:12b,WaldspurgerAM:12}.
While this recent work has focused on stable and efficient phase retrieval from asymptotically few measurements (namely, $N=O(M)$), the present paper focuses on injectivity and algorithmic efficiency with the absolute minimum number of measurements.

In the next section, we construct an ensemble of $N=4M-4$ measurement vectors in $\mathbb{C}^M$ which yield injective intensity measurements. This is the second known injective ensemble of this size (the first is due to Bodmann and Hammen~\cite{BodmannH:13}), and it is conjectured to be the smallest-possible injective ensemble~\cite{BandeiraCMN:13}. 
The same conjecture suggests that $4M-4$ generic measurement vectors yield injectivity (that is, there exists a measure-zero set of ensembles of $4M-4$ vectors such that every ensemble of $4M-4$ vectors outside of this set yields injectivity). 
The following summarizes what is currently known about the so-called ``$4M-4$ conjecture'':

\begin{itemize}
\item The conjecture holds for $M=2,3$~\cite{BandeiraCMN:13}.
\item If $N<4M-2\alpha(M-1)-3$, then $\mathcal{A}$ is not injective~\cite{HeinosaariMW:13}; here, $\alpha(M-1)\leq\log_2 M$ denotes the number of $1$'s in the binary expansion of $M-1$.
\item For each $M\geq2$, there exists an ensemble $\Phi$ of $N=4M-4$ measurement vectors such that $\mathcal{A}$ is injective~\cite{BodmannH:13} (see also Section~2 of this paper).
\item If $N\geq 4M-2$, then $\mathcal{A}$ is injective for generic $\Phi$~\cite{BalanCE:06}.
\end{itemize}

Bodmann and Hammen~\cite{BodmannH:13} leverage the Dirichlet kernel and the Cayley map to prove injectivity of their ensemble, but it is unclear whether phase retrieval is algorithmically feasible from their ensemble. 
By contrast, for the ensemble in this paper, we use basic ideas from harmonic analysis over cyclic groups to devise a corresponding phase retrieval algorithm, and we demonstrate injectivity by proving that the algorithm succeeds.

In Section~3, we devise a theory of ensembles for which the corresponding intensity measurements are ``almost'' injective, that is, $\mathcal{A}^{-1}(\mathcal{A}(x))=\{\omega x:|\omega|=1\}$ for almost every $x$. In this section, we focus on the real case, meaning phase retrieval is up to a global sign factor $\omega=\pm1$, and our approach is inspired by the characterization of injectivity in the real case by Balan, Casazza and Edidin~\cite{BalanCE:06}. 
After characterizing almost injectivity in the real case, we find a particularly satisfying sufficient condition for almost injectivity: that $\Phi$ forms a unit norm tight frame with $M$ and $N$ relatively prime. 
Characterizing almost injectivity in the complex case remains an open problem.

We conclude with Section~4, in which we consider algorithmic phase retrieval in the real case from $N=M+1$ almost injective intensity measurements. 
Specifically, we show that phase retrieval in this case is $\NP$-hard by reduction from the subset sum problem. 
The hardness of phase retrieval in this minimal case suggests a new problem for phase retrieval: What is the smallest $C$ for which there exists a family of ensembles of size $N=CM+o(M)$ such that phase retrieval can be performed in polynomial time?

\section{$4M-4$ injective intensity measurements}

In this section, we provide an ensemble of $4M-4$ measurement vectors which yield injective intensity measurements for $\mathbb{C}^M$.
The vectors in our ensemble are modulated discrete cosine functions, and they are explicitly constructed at the end of this section.
We start here by motivating our construction, specifically by identifying the significance of circular autocorrelation.

Consider the $P$-dimensional complex vector space $\ell(\mathbb{Z}_P):=\{u\colon\mathbb{Z}\rightarrow\mathbb{C}:u[p+P]=u[p],~\forall p\in\mathbb{Z}\}$.
The discrete Fourier basis in $\ell(\mathbb{Z}_P)$ is the sequence of $P$ vectors $\{f_q\}_{q\in\mathbb{Z}_P}$ defined by $f_q[p]:=e^{2\pi ipq/P}$ (the notation ``$q\in\mathbb{Z}_P$'' is taken to mean a set of coset representatives of $\mathbb{Z}$ with respect to the subgroup $P\mathbb{Z}$).
The discrete Fourier transform (DFT) on $\mathbb{Z}_P$ is the analysis operator $F^*\colon\ell(\mathbb{Z}_P)\rightarrow\ell(\mathbb{Z}_P)$ of this basis, with corresponding inverse DFT $(F^*)^{-1}=\frac{1}{P}F$, where
\begin{equation*}
(F^*u)[q]=\langle u,f_q\rangle=\sum_{p\in\mathbb{Z}_P}u[p]e^{-2\pi ipq/P},
\qquad
(Fv)[p]=\sum_{q\in\mathbb{Z}_P}v[q]f_q[p]=\sum_{q\in\mathbb{Z}_P}v[q]e^{2\pi ipq/P}.
\end{equation*}
Now let $T^p\colon\ell(\mathbb{Z}_P)\rightarrow\ell(\mathbb{Z}_P)$ be the translation operator defined by $(T^pu)[p^\prime]:=u[p^\prime-p]$.
The circular autocorrelation of $u$ is then $\operatorname{CirAut}(u)\in\ell(\mathbb{Z}_P)$, defined entrywise by
\begin{equation}
\label{eq.CirAut def}
\operatorname{CirAut}(u)[p]:=\langle u,T^pu\rangle=\sum_{p^\prime\in\mathbb{Z}_P}u[p^\prime]\overline{u[p^\prime-p]}.
\end{equation}
Consider the DFT of a circular autocorrelation:
\begin{align*}
(F^*\operatorname{CirAut}(u))[q]&=\sum_{p\in\mathbb{Z}_P}\sum_{p^\prime\in\mathbb{Z}_P}u[p^\prime]\overline{u[p^\prime-p]}e^{-2\pi ipq/P}\\
&=\sum_{p^\prime\in\mathbb{Z}_P}u[p^\prime]e^{-2\pi ip^\prime q/P}\overline{\bigg(\sum_{p\in\mathbb{Z}_P}u[p^\prime-p]e^{-2\pi i(p^\prime-p)q/P}\bigg)}\\
&=\sum_{p^\prime\in\mathbb{Z}_P}u[p^\prime]e^{-2\pi ip^\prime q/P}\overline{\bigg(\sum_{p^{\prime\prime}\in\mathbb{Z}_P}u[p^{\prime\prime}]e^{-2\pi ip^{\prime\prime}q/P}\bigg)}
=|\langle u,f_q\rangle|^2.
\end{align*}
As such, if one has the intensity measurements $\{|\langle u,f_q\rangle|^2\}_{q\in\mathbb{Z}_P}$, then one may compute the circular autocorrelation $\operatorname{CirAut}(u)$ by applying the inverse DFT.
In order to perform phase retrieval from $\{|\langle u,f_q\rangle|^2\}_{q\in\mathbb{Z}_P}$, it therefore suffices to determine $u$ from $\operatorname{CirAut}(u)$.
This is the motivation for our approach in this section.

To see how to ``invert'' $\operatorname{CirAut}$, let's consider an example.
Take $x=(a,b,c)\in\mathbb{C}^3$ and consider the circular autocorrelation of $x$ as a signal in $\ell(\mathbb{Z}_{3})$:
\begin{align*}
\operatorname{CirAut}(x)=(|a|^2+|b|^2+|c|^2,a\overline{c}+b\overline{a}+c\overline{b},a\overline{b}+b\overline{c}+c\overline{a}).
\end{align*}
Notice that every entry of $\operatorname{CirAut}(x)$ is a nonlinear combination of the entries of $x$, from which it is unclear how to compute the entries of $x$.
To simplify the structure, we pad $x$ with zeros and enforce even symmetry; then the circular autocorrelation of $u:=(2a,b,c,0,0,0,0,c,b)\in\ell(\mathbb{Z}_{9})$ is
\begin{align}
\label{eq.CirAut example}
\operatorname{CirAut}(u)=(4|a|^2&+|b|^2+|c|^2,2\operatorname{Re}(2a\overline{b}+b\overline{c}),|b|^2+4\operatorname{Re}(a\overline{c}),2\operatorname{Re}(b\overline{c}),|c|^2,\notag\\
&|c|^2,2\operatorname{Re}(b\overline{c}),|b|^2+4\operatorname{Re}(a\overline{c}),2\operatorname{Re}(2a\overline{b}+b\overline{c})).
\end{align}
Although it still appears rather complicated, this circular autocorrelation actually lends itself well to recovering the entries of $x$.

Before explaining this further, first note that $9=4(3)-3$, and we can generalize our mapping $x\mapsto u$ by sending vectors in $\mathbb{C}^M$ to members of $\ell(\mathbb{Z}_{4M-3})$.
To make this clear, consider the reversal operator $R\colon\ell(\mathbb{Z}_P)\rightarrow\ell(\mathbb{Z}_P)$ defined by $(Ru)[p]=u[-p]$.
Then given a vector $x\in\mathbb{C}^M$, padding with zeros and enforcing even symmetry is equivalent to embedding $x$ in $\ell(\mathbb{Z}_{4M-3})$ by appending $3M-3$ zeros to $x$ and then taking $u=x+Rx\in\ell(\mathbb{Z}_{4M-3})$.
(From this point forward we use $x$ to represent both the original signal in $\mathbb{C}^M$ and the version of $x$ embedded in $\ell(\mathbb{Z}_{4M-3})$ via zero-padding; the distinction will be clear from context.)
Computing $x\in\mathbb{C}^M$ then reduces to determining the first $M$ entries of $x\in\ell(\mathbb{Z}_{4M-3})$ from $\operatorname{CirAut}(x+Rx)$.
If $x$ is completely real-valued, then this is indeed possible.
For instance, consider the circular autocorrelation \eqref{eq.CirAut example}.
If the entries of $x$ are all real, then this becomes
\begin{align*}
\operatorname{CirAut}(x+Rx)=(4a^2+b^2+c^2,4ab+2bc,b^2+4ac,2bc,c^2,c^2,2bc,b^2+4ac,4ab+2bc).
\end{align*}
Since $\operatorname{CirAut}(x+Rx)[4]=c^2$, we simply take a square root to obtain $c$ up to a sign.
Assuming $c$ is nonzero, we then divide $\operatorname{CirAut}(x+Rx)[3]$ by 2c to determine $b$ up to the same sign.
Then subtracting $b^2$ from $\operatorname{CirAut}(x+Rx)[2]$ and dividing by $4c$ gives $a$ up to the same sign.

From this example, we see that the process of recovering the entries of $x$ from $\operatorname{CirAut}(x+Rx)$ is iterative, working backward through its first $2M-2$ entries.
But what happens if $c$ is zero?
Fortunately, our process doesn't break:
In this case, we have
\begin{align*}
\operatorname{CirAut}(x+Rx)=(4a^2+b^2,4ab,b^2,0,0,0,0,b^2,4ab).
\end{align*}
Thus, we need only start with $\operatorname{CirAut}(x+Rx)[2]$ to determine the remaining entries of $x$ up to a sign.
This observation brings to light the important role of the last nonzero entry of $x$ in our iteration.
The relationship between this coordinate and the entries of $\operatorname{CirAut}(x+Rx)$ will become more rigorous later.

The above example illustrated how a real signal $x$ is determined by $\operatorname{CirAut}(x+Rx)$.
A complex-valued signal, on the other hand, is not completely determined from $\operatorname{CirAut}(x+Rx)$.
Luckily, this can be fixed by introducing a second vector in $\ell(\mathbb{Z}_{4M-3})$ obtained from $x$, and we will demonstrate this later, but for now we focus on $x+Rx$.
To this end, let's first take a closer look at the entries of $\operatorname{CirAut}(x+Rx)$.
Since this circular autocorrelation has even symmetry by construction, we need only consider all entries of $\operatorname{CirAut}(x+Rx)$ up to index $2M-2$.
This leads to the following lemma:

\begin{lemma}
\label{lem.CirAut reduction}
Let $x$ denote an $M$-dimensional complex signal embedded in $\ell(\mathbb{Z}_{4M-3})$ such that $x[p]=0$ for all $p=M,\ldots,4M-4$.
Then $\operatorname{CirAut}(x+Rx)[p]=2\operatorname{Re}\langle x,T^px\rangle+\langle x,RT^{-p}x\rangle$ for all $p=1,\ldots,2M-2$.
\end{lemma}

\begin{proof}
First note that by the definition of the circular autocorrelation in \eqref{eq.CirAut def} we have
\begin{equation*}
\operatorname{CirAut}(x+Rx)[p]
=\langle x+Rx,T^p(x+Rx)\rangle
=2\operatorname{Re}\langle x,T^px\rangle+\langle x,RT^{-p}x\rangle+\langle x,RT^px\rangle.
\end{equation*}
Thus, to complete the proof it suffices to show that $\langle x,RT^px\rangle=0$ for all $p=1,\ldots,2M-2$.
Since $x$ is only nonzero in its first $M$ entries, we have
\begin{equation*}
\langle x,RT^px\rangle=\sum_{p^\prime=0}^{M-1}x[p^\prime]\overline{(RT^px)[p^\prime]}=\sum_{p^\prime=0}^{M-1}x[p^\prime]\overline{(T^px)[-p^\prime]}=\sum_{p^\prime=0}^{M-1}x[p^\prime]\overline{x[-p^\prime-p]},
\end{equation*}
where the summand is zero whenever $-p^\prime-p\notin[0,M-1]$ modulo $4M-3$.
This is equivalent to having $-p$ not lie in the Minkowski sum $p^\prime+[0,M-1]$, and since $p^\prime\in[0,M-1]$ we see that $\langle x,RT^px\rangle=0$ for all $p=1,\ldots,2M-2$.
\end{proof}

As a consequence of Lemma~\ref{lem.CirAut reduction}, the following theorem expresses the entries of $\operatorname{CirAut}(x+Rx)$ in terms of the entries of $x$:

\begin{theorem}
\label{thm.CirAut reduction}
Let $x$ denote an $M$-dimensional complex signal embedded in $\ell(\mathbb{Z}_{4M-3})$ such that $x[p]=0$ for all $p=M,\ldots,4M-4$.
Then we have
\begin{align}
\label{eq.thm 2}
\operatorname{CirAut}(x+Rx)[p]=\begin{cases}
\displaystyle{2\operatorname{Re}\bigg(\sum_{p^\prime=\frac{p+1}{2}}^{M-1}x[p^\prime](\overline{x[p^\prime-p]}+\overline{x[p-p^\prime]})\bigg)}&\;\text{if $p$ is odd}\\
\displaystyle{2\operatorname{Re}\bigg(\sum_{p^\prime=\frac{p}{2}+1}^{M-1}x[p^\prime](\overline{x[p^\prime-p]}+\overline{x[p-p^\prime]})\bigg)+\left|x[\tfrac{p}{2}]\right|^2}&\;\text{if $p$ is even}
\end{cases}
\end{align}
for all $p=1,\ldots,2M-2$.
\end{theorem}

\begin{proof}
We first use Lemma~\ref{lem.CirAut reduction} to get
\begin{align}
\operatorname{CirAut}(x+Rx)[p]
\nonumber
&=2\operatorname{Re}\langle x,T^px\rangle+\langle x,RT^{-p}x\rangle\\
\nonumber
&=2\operatorname{Re}\bigg(\sum_{p^\prime=0}^{M-1}x[p^\prime]\overline{x[p^\prime-p]}\bigg)+\sum_{p^\prime=0}^{M-1}x[p^\prime]\overline{x[p-p^\prime]}\\
\label{eq.ciraut expression}
&=2\operatorname{Re}\bigg(\sum_{p^\prime=p}^{M-1}x[p^\prime]\overline{x[p^\prime-p]}\bigg)+\sum_{p^\prime=\max\{p-(M-1),0\}}^{\min\{p,M-1\}}x[p^\prime]\overline{x[p-p^\prime]},
\end{align}
where the last equality takes into account that the first summand is nonzero only when $p^\prime-p\in[0,M-1]$ and the second summand is nonzero only when $p-p^\prime\in[0,M-1]$, i.e., when $p^\prime\in[p,p+(M-1)]$ and $p^\prime\in[p-(M-1),p]$, respectively.
To continue, we divide our analysis into cases.

For $p=1,\ldots,M-1$, \eqref{eq.ciraut expression} gives
\begin{equation}
\label{eq.thm2 CirAut1}
\operatorname{CirAut}(x+Rx)[p]
=2\operatorname{Re}\bigg(\sum_{p^\prime=p}^{M-1}x[p^\prime]\overline{x[p^\prime-p]}\bigg)+\sum_{p^\prime=0}^{p}x[p^\prime]\overline{x[p-p^\prime]}.
\end{equation}
If $p$ is odd we can then write
\begin{align}
\label{eq.thm2 CirAut1 odd}
\sum_{p^\prime=0}^{p}x[p^\prime]\overline{x[p-p^\prime]}
&=\sum_{p^\prime=0}^{\frac{p-1}{2}}x[p^\prime]\overline{x[p-p^\prime]}+\sum_{p^\prime=\frac{p+1}{2}}^{p}x[p^\prime]\overline{x[p-p^\prime]}\notag\\
&=\sum_{p^{\prime\prime}=\frac{p+1}{2}}^{p}x[p-p^{\prime\prime}]\overline{x[p^{\prime\prime}]}+\sum_{p^\prime=\frac{p+1}{2}}^{p}x[p^\prime]\overline{x[p-p^\prime]}
=2\operatorname{Re}\bigg(\sum_{p^\prime=\frac{p+1}{2}}^{p}x[p^\prime]\overline{x[p-p^\prime]}\bigg),
\end{align}
while if $p$ is even we similarly write
\begin{equation}
\label{eq.thm2 CirAut1 even}
\sum_{p^\prime=0}^{p}x[p^\prime]\overline{x[p-p^\prime]}
=2\operatorname{Re}\bigg(\sum_{p^\prime=\frac{p}{2}+1}^{p}x[p^\prime]\overline{x[p-p^\prime]}\bigg)+\left|x\big[\tfrac{p}{2}\big]\right|^2.
\end{equation}
Substituting \eqref{eq.thm2 CirAut1 odd} and \eqref{eq.thm2 CirAut1 even} into \eqref{eq.thm2 CirAut1} then gives \eqref{eq.thm 2}.

For the remaining case, $p=M,\ldots,2M-2$ and \eqref{eq.ciraut expression} gives
\begin{equation}
\label{eq.thm2 CirAut2}
\operatorname{CirAut}(x+Rx)[p]
=\sum_{p^\prime=p-(M-1)}^{M-1}x[p^\prime]\overline{x[p-p^\prime]}.
\end{equation}
Similar to the previous case, taking $p$ to be odd yields
\begin{equation}
\label{eq.thm2 CirAut2 odd}
\sum_{p^\prime=p-(M-1)}^{M-1}x[p^\prime]\overline{x[p-p^\prime]}
=2\operatorname{Re}\bigg(\sum_{p^\prime=\frac{p+1}{2}}^{M-1}x[p^\prime]\overline{x[p-p^\prime]}\bigg),
\end{equation}
while taking $p$ to be even yields
\begin{equation}
\label{eq.thm2 CirAut2 even}
\sum_{p^\prime=p-(M-1)}^{M-1}x[p^\prime]\overline{x[p-p^\prime]}
=2\operatorname{Re}\bigg(\sum_{p^\prime=\frac{p}{2}+1}^{M-1}x[p^\prime]\overline{x[p-p^\prime]}\bigg)+\left|x\big[\tfrac{p}{2}\big]\right|^2,
\end{equation}
and substituting \eqref{eq.thm2 CirAut2 odd} and \eqref{eq.thm2 CirAut2 even} into \eqref{eq.thm2 CirAut2} also gives \eqref{eq.thm 2}.
\end{proof}

Notice \eqref{eq.thm 2} shows that each member of $\{\operatorname{CirAut}(x+Rx)[p]\}_{p=1}^{2M-2}$ can be written as a combination of the first $M$ entries of $x$, but only those at or beyond the $\lceil\frac{p}{2}\rceil$th index.
As such, the index of the last nonzero entry of $x$ is closely related to that of the last nonzero entry of $\{\operatorname{CirAut}(x+Rx)[p]\}_{p=1}^{2M-2}$.
This corresponds to our observation earlier in the case of $x\in\mathbb{R}^3$ where the third coordinate was assumed to be zero.
We identify the relationship between the locations of these nonzero entries in the following lemma:

\begin{lemma}
\label{lem.CirAut last nonzero entry even index}
Let $x$ denote an $M$-dimensional complex signal embedded in $\ell(\mathbb{Z}_{4M-3})$ such that $x[p]=0$ for all $p=M,\ldots,4M-4$.
Then the last nonzero entry of $\{\operatorname{CirAut}(x+Rx)[p]\}_{p=0}^{2M-2}$ has index $p=2q$, where $q$ is the index of the last nonzero entry of $x$.
\end{lemma}

\begin{proof}
If $q\geq1$, then \eqref{eq.thm 2} gives that $\operatorname{CirAut}(x+Rx)[2q]=|x[q]|^2\neq0$.
Note that since $x[p']=0$ for every $p'>q$, \eqref{eq.thm 2} also gives that $\operatorname{CirAut}(x+Rx)[p]=0$ for every $p>2q$.
For the remaining case where $q=0$, \eqref{eq.thm 2} immediately gives that $\operatorname{CirAut}(x+Rx)[p]=0$ for every $p\geq1$.
To show that $\operatorname{CirAut}(x+Rx)[0]\neq0$ in this case, we apply the definition of circular autocorrelation \eqref{eq.CirAut def}:
\begin{equation*}
\operatorname{CirAut}(x+Rx)[0]
=\langle x+Rx,x+Rx\rangle
=\|x+Rx\|^2
=|2x[0]|^2
\neq0,
\end{equation*}
where the last equality uses the fact that $x$ is only supported at $0$ since $q=0$.
\end{proof}

As previously mentioned, we are unable to recover the entries of a complex signal $x$ solely from $\operatorname{CirAut}(x+Rx)$.
One way to address this is to rotate the entries of $x$ in the complex plane and also take the circular autocorrelation of this modified signal.
If we rotate by an angle which is not an integer multiple of $\pi$, this will produce new entries which are linearly independent from the corresponding entries of $x$ when viewed as vectors in the complex plane.
As we will see, the problem of recovering the entries of $x$ then reduces to solving a linear system.

Take any $(4M-3)\times(4M-3)$ diagonal modulation operator $E$ whose diagonal entries $\{\omega_k\}_{k=0}^{4M-4}$ are of unit modulus satisfying $\omega_j\overline{\omega_k}\notin\mathbb{R}$ for all $j\neq k$ and consider the new vector $Ex\in\ell(\mathbb{Z}_{4M-3})$.
Then Theorem~\ref{thm.CirAut reduction} gives
\begin{align}
\label{eq.thm 2 for Mx}
&\operatorname{CirAut}(Ex+REx)[p]\notag\\
&\qquad\qquad=\begin{cases}
\displaystyle{2\operatorname{Re}\bigg(\sum_{p^\prime=\frac{p+1}{2}}^{M-1}\omega_{p^\prime}x[p^\prime](\overline{\omega_{p^\prime-p}}\overline{x[p^\prime-p]}+\overline{\omega_{p-p^\prime}}\overline{x[p-p^\prime]})\bigg)}&\;\text{if $p$ is odd}\\
\displaystyle{2\operatorname{Re}\bigg(\sum_{p^\prime=\frac{p}{2}+1}^{M-1}\omega_{p^\prime}x[p^\prime](\overline{\omega_{p^\prime-p}}\overline{x[p^\prime-p]}+\overline{\omega_{p-p^\prime}}\overline{x[p-p^\prime]})\bigg)+\left|x[\tfrac{p}{2}]\right|^2}&\;\text{if $p$ is even}
\end{cases}
\end{align}
for all $p=1,\ldots,2M-2$.
We will see that \eqref{eq.thm 2} and \eqref{eq.thm 2 for Mx} together allow us to solve for the entries of $x$ (up to a global phase factor) by working iteratively backward through the entries of $\operatorname{CirAut}(x+Rx)$ and $\operatorname{CirAut}(Ex+REx)$.
As alluded to earlier, each entry index forms a linear system which can be solved using the following lemma:

\begin{lemma}
\label{lem.modulus isolation}
Let $a,b\in\mathbb{C}\setminus\{0\}$ and $\omega\in\mathbb{C}\setminus\mathbb{R}$ with $|\omega|=1$.
Then
\begin{equation}
\label{eq.modulus isolation}
b=\frac{i}{\overline{a}\operatorname{Im}(\omega)}\big(\operatorname{Re}(\omega a\overline{b})-\omega\operatorname{Re}(a\overline{b})\big).
\end{equation}
\end{lemma}

\begin{proof}
Define $\theta:=\operatorname{arg}(\omega)$ and $\phi:=\operatorname{arg}(a\overline{b})$.
Then $\theta+\phi\equiv\operatorname{arg}(\omega ab)\bmod2\pi$ and
\begin{equation*}
\cos(\phi)=\frac{\operatorname{Re}(a\overline{b})}{|a\overline{b}|},
\qquad
\sin(\phi)=\frac{\operatorname{Im}(a\overline{b})}{|a\overline{b}|},
\qquad
\cos(\theta+\phi)=\frac{\operatorname{Re}(\omega a\overline{b})}{|\omega a\overline{b}|}.
\end{equation*}
With this, we apply a trigonometric identity to obtain
\begin{equation*}
\operatorname{Re}(\omega a\overline{b})
=|\omega a\overline{b}|\cos(\theta+\phi)
=|a\overline{b}|\left(\cos(\theta)\cos(\phi)-\sin(\theta)\sin(\phi)\right)
=\cos(\theta)\operatorname{Re}(a\overline{b})-\sin(\theta)\operatorname{Im}(a\overline{b}).
\end{equation*}
Since $\omega\in\mathbb{C}\setminus\mathbb{R}$, then $\sin(\theta)$ is necessarily nonzero, and so we can isolate $\operatorname{Im}(a\overline{b})$ in the above equation.
We then use this expression for $\operatorname{Im}(a\overline{b})$ to solve for $b$:
\begin{equation*}
b
=\frac{\overline{a}b}{~\overline{a}~}
=\frac{1}{~\overline{a}~}\big(\operatorname{Re}(a\overline{b})-i\operatorname{Im}(a\overline{b})\big)
=\frac{i}{\overline{a}\sin(\theta)}\big(\operatorname{Re}(\omega a\overline{b})-e^{i\theta}\operatorname{Re}(a\overline{b})\big).
\qedhere
\end{equation*}
\end{proof}

We now use this lemma to describe how to recover $x$ up to global phase.
By Lemma~\ref{lem.CirAut last nonzero entry even index}, the last nonzero entry of $\{\operatorname{CirAut}(x+Rx)[p]\}_{p=0}^{2M-2}$ has index $p=2q$, where $q$ indexes the last nonzero entry of $x$.
As such, we know that $x[k]=0$ for every $k>q$, and $x[q]$ can be estimated up to a phase factor ($\hat{x}[q]=e^{i\psi}x[q]$) by taking the square root of $\operatorname{CirAut}(x+Rx)[2q]=|x[q]|^2$ (we will verify this soon, but this corresponds to the examples we have seen so far).
Next, if we know $\operatorname{Re}(x[q]\overline{x[k]})$ and $\operatorname{Re}(\omega_q\overline{\omega_k}x[q]\overline{x[k]})$ for some $k<q$, then we can use these to estimate $x[k]$:
\begin{equation}
\label{eq.x[k] solution}
\hat{x}[k]
:=\frac{i}{\overline{\hat{x}[q]}\operatorname{Im}(\omega_q\overline{\omega_k})}\left(\operatorname{Re}(\omega_q\overline{\omega_k}x[q]\overline{x[k]})-\omega_q\overline{\omega_k}\operatorname{Re}(x[q]\overline{x[k]})\right)
=e^{i\psi}x[k],
\end{equation}
where the last equality follows from substituting $a=x[q]$, $b=x[k]$ and $\omega=\omega_q\overline{\omega_k}$ into \eqref{eq.modulus isolation}.
Overall, once we know $x[q]$ up to phase, then we can find $x[k]$ relative to this same phase for each $k=0,\ldots,q-1$, provided we know $\operatorname{Re}(x[q]\overline{x[k]})$ and $\operatorname{Re}(\omega_q\overline{\omega_k}x[q]\overline{x[k]})$ for these $k$'s.
Thankfully, these values can be determined from the entries of $\operatorname{CirAut}(x+Rx)$ and $\operatorname{CirAut}(Ex+REx)$:

\begin{theorem}
\label{thm.CirAut recovers x}
Let $x$ denote an $M$-dimensional complex signal embedded in $\ell(\mathbb{Z}_{4M-3})$ such that $x[p]=0$ for all $p=M,\ldots,4M-4$ and $E$ be a $(4M-3)\times(4M-3)$ diagonal modulation operator with diagonal entries $\{\omega_k\}_{k=0}^{4M-4}$ satisfying $|\omega_k|=1$ for all $k=0,\ldots,4M-4$ and $\omega_j\overline{\omega_k}\notin\mathbb{R}$ for all $j\neq k$.
Then $x$ can be recovered up to a global phase factor from $\operatorname{CirAut}(x+Rx)$ and $\operatorname{CirAut}(Ex+REx)$.
\end{theorem}

\begin{proof}
Letting $q$ denote the last nonzero entry of $x$, it suffices to estimate $\{x[k]\}_{k=0}^q$ up to a global phase factor.
To this end, recall from Lemma~\ref{lem.CirAut last nonzero entry even index} that the last nonzero entry of $\{\operatorname{CirAut}(x+Rx)[p]\}_{p=0}^{2M-2}$ has index $p=2q$.
If $q=0$, then we have already seen that $\operatorname{CirAut}(x+Rx)[0]=4|x[0]|^2$.
Since there exists $\psi\in[0,2\pi)$ such that $x[0]=e^{-i\psi}|x[0]|$, we may take $\hat{x}[0]:=\frac{1}{2}\sqrt{\operatorname{CirAut}(x+Rx)[0]}=|x[0]|=e^{i\psi}x[0]$.
Otherwise $q\in[1,M-1]$, and \eqref{eq.thm 2} gives
\begin{align*}
\operatorname{CirAut}(x+Rx)[2q]&=\left|x[q]\right|^2+2\operatorname{Re}\bigg(\sum_{p^\prime=q+1}^{M-1}x[p^\prime](\overline{x[p^\prime-2q]}+\overline{x[2q-p^\prime]})\bigg)=\left|x[q]\right|^2.
\end{align*}
Thus, taking $\hat{x}[q]:=\sqrt{\operatorname{CirAut}(x+Rx)[2q]}=|x[q]|$ gives us $\hat{x}[q]=e^{i\psi}x[q]$ for some $\psi\in[0,2\pi)$.

In the case where $q=1$, all that remains to determine is $\hat{x}[0]$, a calculation which we save for the end of the proof.
For now, suppose $q\geq2$.
Since we already know $\hat{x}[q]=e^{i\psi}x[q]$, we would like to determine $\hat{x}[k]$ for $k=1,\ldots,q-1$.
To this end, take $r\in[0,q-2]$ and suppose we have $\hat{x}[k]=e^{i\psi}x[k]$ for all $k=q-r,\ldots,q$.
If we can obtain $\hat{x}[q-(r+1)]$ up to the same phase from this information, then working iteratively from $r=0$ to $r=q-2$ will give us $\hat{x}[k]$ up to global phase for all but the zeroth entry (which we address later).
Note when $r$ is even, \eqref{eq.thm 2} gives
\begin{align*}
\operatorname{CirAut}(x+Rx)[2q-(r+1)]
&=2\operatorname{Re}\bigg(\sum_{p^\prime=q-\frac{r}{2}}^{q}x[p^\prime](\overline{x[p^\prime-(2q-(r+1))]}+\overline{x[(2q-(r+1))-p^\prime]})\bigg)\\
&=2\operatorname{Re}\left(x[q]\overline{x[q-(r+1)]}\right)+2\sum_{p^\prime=q-\frac{r}{2}}^{q-1}\operatorname{Re}\left(x[p^\prime]\overline{x[(2q-(r+1))-p^\prime]}\right),
\end{align*}
where the last equality follows from the observation that $p^\prime-(2q-(r+1))\leq-q+(r+1)\leq-1$ over the range of the sum, meaning $x[p^\prime-(2q-(r+1))]=0$ throughout the sum.
Similarly when $r$ is odd, \eqref{eq.thm 2} gives
\begin{align*}
&\operatorname{CirAut}(x+Rx)[2q-(r+1)]\\
&\qquad\qquad=2\operatorname{Re}\left(x[q]\overline{x[q-(r+1)]}\right)+2\sum_{p^\prime=q-\frac{r-1}{2}}^{q-1}\operatorname{Re}\left(x[p^\prime]\overline{x[(2q-(r+1))-p^\prime]}\right)+\left|x\big[q-\tfrac{r+1}{2}\big]\right|^2.
\end{align*}
In either case, we can isolate $\operatorname{Re}(x[q]\overline{x[q-(r+1)]})$ to get an expression in terms of $\operatorname{CirAut}(x+Rx)[2q-(r+1)]$ and other terms of the form $\operatorname{Re}(x[k]\overline{x[k']})$ or $|x[k]|^2$ for $k,k'\in[q-r,q-1]$.
By the induction hypothesis, we have $\hat{x}[k]=e^{i\psi}x[k]$ for $k=q-r,\ldots,q-1$, and so we can use these estimates to determine these other terms:
\begin{equation*}
\operatorname{Re}(\hat{x}[k]\overline{\hat{x}[k']})
=\operatorname{Re}(e^{i\psi}x[k]\overline{e^{i\psi}x[k']})
=\operatorname{Re}(x[k]\overline{x[k']}),
\qquad
|\hat{x}[k]|^2
=|e^{i\psi}x[k]|^2
=|x[k]|^2.
\end{equation*}
As such, we can use $\operatorname{CirAut}(x+Rx)[2q-(r+1)]$ along with the higher-indexed estimates $\hat{x}[k]$ to determine $\operatorname{Re}(x[q]\overline{x[q-(r+1)]})$.
Similarly, we can use $\operatorname{CirAut}(Ex+REx)[2q-(r+1)]$ along with the higher-indexed estimates $\hat{x}[k]$ to determine $\operatorname{Re}(\omega_q\overline{\omega_{(q-(r+1))}}x[q]\overline{x[q-(r+1)]})$.
We then plug these into \eqref{eq.x[k] solution}, along with the estimate $\hat{x}[q]=e^{i\psi}x[q]$ (which is also available by the induction hypothesis), to get $\hat{x}[2q-(r+1)]=e^{i\psi}x[2q-(r+1)]$.

At this point, we have determined $\{x[k]\}_{k=1}^{q}$ up to a global phase factor whenever $q\geq1$, and so it remains to find $\hat{x}[0]$.
For this, note that when $q$ is odd, \eqref{eq.thm 2} gives
\begin{equation*}
\operatorname{CirAut}(x+Rx)[q]
=4\operatorname{Re}(x[q]\overline{x[0]})+2\sum_{p^\prime=\frac{q+1}{2}}^{q-1}\operatorname{Re}\left(x[p^\prime]\overline{x[q-p^\prime]}\right),
\end{equation*}
while for even $q$, we have
\begin{align*}
\operatorname{CirAut}(x+Rx)[q]
=4\operatorname{Re}(x[q]\overline{x[0]})+2\sum_{p^\prime=\frac{q}{2}+1}^{q-1}\operatorname{Re}\left(x[p^\prime]\overline{x[q-p^\prime]}\right)+\left|x\big[\tfrac{q}{2}\big]\right|^2.
\end{align*}
As before, isolating $\operatorname{Re}(x[q]\overline{x[0]})$ in either case produces an expression in terms of $\operatorname{CirAut}(x+Rx)[q]$ and other terms of the form $\operatorname{Re}(x[k]\overline{x[k']})$ or $|x[k]|^2$ for $k,k'\in[1,q-1]$.
These other terms can be calculated using the estimates $\{\hat{x}[k]\}_{k=1}^{q-1}$, and so we can also calculate $\operatorname{Re}(x[q]\overline{x[0]})$ from $\operatorname{CirAut}(x+Rx)[q]$.
Similarly, we can calculate $\operatorname{Re}(\omega_q\overline{\omega_0}x[q]\overline{x[0]})$ from $\{\hat{x}[k]\}_{k=1}^{q-1}$ and $\operatorname{CirAut}(Ex+REx)[q]$, and plugging these into \eqref{eq.x[k] solution} along with $\hat{x}[q]$ produces the estimate $\hat{x}[0]=e^{i\psi}x[0]$.
\end{proof}

Theorem~\ref{thm.CirAut recovers x} establishes that it is possible to recover a signal $x\in\mathbb{C}^M$ up to a global phase factor from $\{\operatorname{CirAut}(x+Rx)\}_{q=0}^{2M-2}$ and $\{\operatorname{CirAut}(Ex+REx)\}_{q=0}^{2M-2}$.
We now return to how these circular autocorrelations relate to intensity measurements.
Recall that the DFT of the circular autocorrelation is the modulus squared of the DFT of the original signal: $(F^*\operatorname{CirAut}(u))[q]=|(F^*u)[q]|^2$.
Also note that the DFT commutes with the reversal operator:
\begin{equation*}
(F^*Ru)[q]
=\sum_{p\in\mathbb{Z}_P}u[-p]e^{-2\pi ipq/P}
=\sum_{p^\prime\in\mathbb{Z}_P}u[p^\prime]e^{-2\pi ip^\prime(-q)/P}
=(F^*u)[-q]
=(RF^*u)[q].
\end{equation*}
With this, we can express $\operatorname{CirAut}(x+Rx)$ in terms of intensity measurements with a particular ensemble:
\begin{align*}
(F^*\operatorname{CirAut}(x+Rx))[q]
&=|(F^*(x+Rx))[q]|^2\\
&=|(F^*x)[q]+(F^*Rx)[q]|^2
=|(F^*x)[q]+(F^*x)[-q]|^2
=|\langle x,f_q+f_{-q}\rangle|^2.
\end{align*}
Defining the $q$th \textit{discrete cosine function} $c_q\in\ell(\mathbb{Z}_{4M-3})$ by
\begin{equation*}
c_q[p]:=2\cos\left(\tfrac{2\pi pq}{4M-3}\right)=e^{2\pi ipq/(4M-3)}+e^{-2\pi ipq/(4M-3)}=(f_q+f_{-q})[p],
\end{equation*}
this means that $(F^*\operatorname{CirAut}(x+Rx))[q]=|\langle x,c_q\rangle|^2$ for all $q\in\mathbb{Z}_{4M-3}$.
Similarly, if we take the modulation matrix $E$ to have diagonal entries $\omega_k=e^{2\pi ik/(2M-1)}$ for all $k=0,\ldots,4M-4$, we find
\begin{equation*}
(F^*\operatorname{CirAut}(Ex+REx))[q]
=|\langle Ex,c_q\rangle|^2
=|\langle x,E^*c_q\rangle|^2.
\end{equation*}
Thus, coupling the DFT with Theorem~\ref{thm.CirAut recovers x} allows us to recover the signal $x$ from $4M-2$ intensity measurements, namely with the ensemble $\{c_q\}_{q=0}^{2M-2}\cup\{E^*c_q\}_{q=0}^{2M-2}$.
Note that since $x\in\ell(\mathbb{Z}_{4M-3})$ is actually a zero-padded version of $x\in\mathbb{C}^M$, we may view $c_q$ and $E^*c_q$ as members of $\mathbb{C}^M$ by discarding the entries indexed by $p=M,\ldots,4M-4$.

Considering this section promised phase retrieval from only $4M-4$ intensity measurements, we must somehow find a way to discard two of these $4M-2$ measurement vectors.
To do this, first note that
\begin{align*}
\operatorname{CirAut}(Ex+REx)[0]
&=\|Ex+REx\|^2\\
&=\sum_{k\in\mathbb{Z}_{4M-3}}\left|e^{2\pi ik/(2M-1)}x[k]+e^{2\pi i(-k)/(2M-1)}x[-k]\right|^2\\
&=\sum_{k=-(2M-2)}^{-1}\left|e^{2\pi i(-k)/(2M-1)}x[-k]\right|^2+|2x[0]|^2+\sum_{k=1}^{2M-2}\left|e^{2\pi ik/(2M-1)}x[k]\right|^2\\
&=\|x+Rx\|^2\\
&=\operatorname{CirAut}(x+Rx)[0].
\end{align*}
Moreover, we have
\begin{align*}
\operatorname{CirAut}(Ex+REx)[2M-2]&=\sum_{k\in\mathbb{Z}_{4M-3}}(Ex+REx)[k]\overline{(Ex+REx)[k-(2M-2)]}\\
&=(Ex+REx)[M-1]\overline{(Ex+REx)[-(M-1)]}\\
&=(Ex+REx)[M-1]\overline{(Ex+REx)[M-1]},
\end{align*}
where the last equality is by even symmetry.
Since $x$ is only supported on $k=0,\ldots,M-1$, we then have
\begin{align*}
\operatorname{CirAut}(Ex+REx)[2M-2]
&=|(Ex+REx)[M-1]|^2\\
&=\left|e^{2\pi i(M-1)/(2M-1)}x[M-1]+e^{-2\pi i(M-1)/(2M-1)}x[-(M-1)]\right|^2\\
&=\left|e^{2\pi i(M-1)/(2M-1)}x[M-1]\right|^2
=|x[M-1]|^2
=\operatorname{CirAut}(x+Rx)[2M-2].
\end{align*}
Furthermore, the even symmetry of the circular autocorrelation also gives
\begin{align*}
\operatorname{CirAut}(Ex+REx)[-(2M-2)]&=\operatorname{CirAut}(Ex+REx)[2M-2]\\
&=\operatorname{CirAut}(x+Rx)[2M-2]
=\operatorname{CirAut}(x+Rx)[-(2M-2)].
\end{align*}
These redundancies between $\operatorname{CirAut}(x+Rx)$ and $\operatorname{CirAut}(Ex+REx)$ indicate that we might be able to remove measurement vectors from our ensemble while maintaining our ability to perform phase retrieval.
The following theorem confirms this suspicion:

\begin{theorem}
\label{thm.4M-4 intensity measurements}
Let $c_q\in\mathbb{C}^M$ be the truncated discrete cosine function defined by $c_q[p]:=2\cos(\frac{2\pi pq}{4M-3})$ for all $p=0,\ldots,M-1$, and let $E$ be the $M\times M$ diagonal modulation operator with diagonal entries $\omega_k=e^{2\pi ik/(2M-1)}$ for all $k=0,\ldots,M-1$.
Then the intensity measurement mapping $\mathcal{A}\colon\mathbb{C}^M/\mathbb{T}\rightarrow\mathbb{R}^{4M-4}$ defined by $\mathcal{A}(x):=\{|\langle x,c_q\rangle|^2\}_{q=0}^{2M-2}\cup\{|\langle x,E^*c_q\rangle|^2\}_{q=1}^{2M-3}$ is injective.
\end{theorem}

\begin{proof}
Since Theorem~\ref{thm.CirAut recovers x} allows us to reconstruct any $x\in\mathbb{C}^M$ up to a global phase factor from the entries of $\operatorname{CirAut}(x+Rx)$ and $\operatorname{CirAut}(Ex+REx)$, it suffices to show that the intensity measurements $\{|\langle x,c_q\rangle|^2\}_{q=0}^{2M-2}\cup\{|\langle x,E^*c_q\rangle|^2\}_{q=1}^{2M-3}$ allow us to recover the entries of these circular autocorrelations.
To this end, recall that
\begin{align*}
\operatorname{CirAut}(x+Rx)=(F^*)^{-1}\{|\langle x,c_q\rangle|^2\}_{q\in\mathbb{Z}_{4M-3}},
\quad
\operatorname{CirAut}(Ex+REx)=(F^*)^{-1}\{|\langle x,E^*c_q\rangle|^2\}_{q\in\mathbb{Z}_{4M-3}}.
\end{align*}
Since we have $\{|\langle x,c_q\rangle|^2\}_{q=0}^{2M-2}$, we can exploit even symmetry to determine the rest of $\{|\langle x,c_q\rangle|^2\}_{q\in\mathbb{Z}_{4M-3}}$, and then apply the inverse DFT to get $\operatorname{CirAut}(x+Rx)$.
Moreover, by the previous discussion, we also obtain the $0$, $2M-2$, and $-(2M-2)$ entries of $\operatorname{CirAut}(Ex+REx)$ from the corresponding entries of $\operatorname{CirAut}(x+Rx)$.
Organize this information about $\operatorname{CirAut}(Ex+REx)$ into a vector $w\in\ell(\mathbb{Z}_{4M-3})$ whose $0$, $2M-2$, and $-(2M-2)$ entries come from $\operatorname{CirAut}(Ex+REx)$ and whose remaining entries are populated by even symmetry from $\{|\langle x,E^*c_q\rangle|^2\}_{q=1}^{2M-3}$.
We can express $w$ as a matrix-vector product $w=A\{|\langle x,E^*c_q\rangle|^2\}_{q\in\mathbb{Z}_{4M-3}}$, where $A$ is the identity matrix with the $0$, $2M-2$, and $-(2M-2)$ rows replaced by the corresponding rows of the inverse DFT matrix.
To complete the proof, it suffices to show that the matrix $A$ is invertible, since this would imply $\operatorname{CirAut}(Ex+REx)=(F^*)^{-1}A^{-1}w$.

Using the cofactor expansion, note that $\det(A)$ reduces to a determinant of a 3$\times$3 submatrix of $(F^*)^{-1}$.
Specifically, letting $\theta:=2\pi(2M-2)^2/(4M-3)$ we have
\begin{align*}
\det(A)
=\det\left(\left[\begin{array}{ccc}
1&1&1\\
1&e^{i\theta}&e^{-i\theta}\\
1&e^{-i\theta}&e^{i\theta}
\end{array}\right]\right)
&=(e^{2i\theta}-e^{-2i\theta})-(e^{i\theta}-e^{-i\theta})+(e^{-i\theta}-e^{i\theta})\\
&=(e^{i\theta}+e^{-i\theta}-2)(e^{i\theta}-e^{-i\theta})
=4i(\cos(\theta)-1)\sin(\theta),
\end{align*}
and so $A$ is invertible if and only if $\cos(\theta)-1\neq0$ and $\sin(\theta)\neq0$.
This equivalent to having $\pi$ not divide $\theta$, and indeed, the ratio
\begin{align*}
\frac{\theta}{\pi}
=\frac{2(2M-2)^2}{4M-3}
=2M-\frac{5}{2}+\frac{1}{2(4M-3)}
\end{align*}
is not an integer because $M\geq2$.
As such, $A$ is invertible.
\end{proof}

We conclude this section by summarizing our measurement design and phase retrieval procedure:\\

\noindent
\textbf{Measurement design}
\begin{itemize}
\item
Define the $q$th truncated discrete cosine function $c_q:=\{2\cos(\frac{2\pi pq}{4M-3})\}_{p=0}^{M-1}$
\item
Define the $M\times M$ diagonal matrix $E$ with entries $\omega_k:=e^{2\pi ik/(2M-1)}$ for all $k=0,\ldots,M-1$
\item
Take $\Phi:=\{c_q\}_{q=0}^{2M-2}\cup\{E^*c_q\}_{q=1}^{2M-3}$
\end{itemize}

\noindent
\textbf{Phase retrieval procedure}
\begin{itemize}
\item
Calculate $\{|\langle x,c_q\rangle|^2\}_{q\in\mathbb{Z}_{4M-3}}$ from $\{|\langle x,c_q\rangle|^2\}_{q=0}^{2M-2}$ by even extension
\item
Calculate $\operatorname{CirAut}(x+Rx)=(F^*)^{-1}\{|\langle x,c_q\rangle|^2\}_{q\in\mathbb{Z}_{4M-3}}$
\item
Define $w\in\ell(\mathbb{Z}_{4M-3})$ so that its $0$, $2M-2$, and $-(2M-2)$ entries are the corresponding entries in $\operatorname{CirAut}(x+Rx)$ and its remaining entries are populated by even symmetry from $\{|\langle x,E^*c_q\rangle|^2\}_{q=1}^{2M-3}$
\item
Define $A$ to be the identity matrix with the $0$, $2M-2$, and $-(2M-2)$ rows replaced by the corresponding rows of the inverse DFT matrix $(F^*)^{-1}$
\item
Calculate $\operatorname{CirAut}(Ex+REx)=(F^*)^{-1}A^{-1}w$
\item
Recover $x$ up to global phase from $\operatorname{CirAut}(x+Rx)$ and $\operatorname{CirAut}(Ex+REx)$ using the process described in the proof of Theorem~\ref{thm.CirAut recovers x}
\end{itemize}

\section{Almost injectivity}

While $4M+o(M)$ measurements are necessary and generically sufficient for injectivity in the complex case, you can save a factor of $2$ in the number of measurements if you are willing to slightly weaken the desired notion of injectivity~\cite{BalanCE:06,FlammiaSC:05}.
To be explicit, we start with the following definition:

\begin{definition}
\label{def.almost injective}
Consider $\Phi=\{\varphi_n\}_{n=1}^N\subseteq\mathbb{R}^M$.
The intensity measurement mapping $\mathcal{A}\colon\mathbb{R}^M/\{\pm1\}\rightarrow\mathbb{R}^N$ defined by $(\mathcal{A}(x))(n):=|\langle x,\varphi_n\rangle|^2$ is said to be \textit{almost injective} if $\mathcal{A}^{-1}(\mathcal{A}(x))=\{\pm x\}$ for almost every $x\in\mathbb{R}^M$.
\end{definition}

The above definition specifically treats the real case, but it can be similarly defined for the complex case in the obvious way.
For the complex case, it is known that $2M$ measurements are necessary for almost injectivity~\cite{FlammiaSC:05}, and that $2M$ generic measurements suffice~\cite{BalanCE:06}; this is the factor-of-$2$ savings mentioned above.
For the real case, it is also known how many measurements are necessary and generically sufficient for almost injectivity: $M+1$~~\cite{BalanCE:06}.
Like the complex case, this is also a factor-of-$2$ savings from the injectivity requirement: $2M-1$.
This requirement for injectivity in the real case follows from the following result from~~\cite{BalanCE:06}, which we prove here because the proof is short and inspires the remainder of this section:

\begin{theorem}
\label{thm.complement property characterization}
Consider $\Phi=\{\varphi_n\}_{n=1}^N\subseteq\mathbb{R}^M$ and the intensity measurement mapping $\mathcal{A}\colon\mathbb{R}^M/\{\pm1\}\rightarrow\mathbb{R}^N$ defined by $(\mathcal{A}(x))(n):=|\langle x,\varphi_n\rangle|^2$. 
Then $\mathcal{A}$ is injective if and only if for every $S\subseteq\{1,\ldots,N\}$, either $\{\varphi_n\}_{n\in S}$ or $\{\varphi_n\}_{n\in S^\mathrm{c}}$ spans $\mathbb{R}^M$.
\end{theorem}

\begin{proof}
We will prove both directions by obtaining the contrapositives.

($\Rightarrow$) 
Assume there exists $S\subseteq\{1,\ldots,N\}$ such that neither $\{\varphi_n\}_{n\in S}$ nor $\{\varphi_n\}_{n\in S^\mathrm{c}}$ spans $\mathbb{R}^M$. 
This implies that there are nonzero vectors $u,v\in\mathbb{R}^M$ such that $\langle u,\varphi_n\rangle=0$ for all $n\in S$ and $\langle v,\varphi_n\rangle=0$ for all $n\in S^\mathrm{c}$.
For each $n$, we then have
\begin{equation*}
|\langle u\pm v,\varphi_n\rangle|^2
=|\langle u,\varphi_n\rangle|^2\pm2\operatorname{Re}\langle u,\varphi_n\rangle\overline{\langle v,\varphi_n\rangle}+|\langle v,\varphi_n\rangle|^2
=|\langle u,\varphi_n\rangle|^2+|\langle v,\varphi_n\rangle|^2.
\end{equation*}
Since $|\langle u+v,\varphi_n\rangle|^2=|\langle u-v,\varphi_n\rangle|^2$ for every $n$, we have $\mathcal{A}(u+v)=\mathcal{A}(u-v)$. 
Moreover, $u$ and $v$ are nonzero by assumption, and so $u+v\neq\pm(u-v)$.

($\Leftarrow$) 
Assume that $\mathcal{A}$ is not injective. 
Then there exist vectors $x,y\in\mathbb{R}^M$ such that $x\neq\pm y$ and $\mathcal{A}(x)=\mathcal{A}(y)$. 
Taking $S:=\{n:\langle x,\varphi_n\rangle=-\langle y,\varphi_n\rangle\}$, we have $\langle x+y,\varphi_n\rangle=0$ for every $n\in S$.
Otherwise when $n\in S^\mathrm{c}$, we have $\langle x,\varphi_n\rangle=\langle y,\varphi_n\rangle$ and so $\langle x-y,\varphi_n\rangle=0$.
Furthermore, both $x+y$ and $x-y$ are nontrivial since $x\neq\pm y$, and so neither $\{\varphi_n\}_{n\in S}$ nor $\{\varphi_n\}_{n\in S^\mathrm{c}}$ spans $\mathbb{R}^M$.
\end{proof}

Similar to the above result, in this section, we characterize ensembles of measurement vectors which yield almost injective intensity measurements, and similar to the above proof, the basic idea behind our analysis is to consider sums and differences of signals with identical intensity measurements.
Our characterization starts with the following lemma:

\begin{lemma}
\label{lem.almost injective and Minkowski sum}
Consider $\Phi=\{\varphi_n\}_{n=1}^N\subseteq\mathbb{R}^M$ and the intensity measurement mapping $\mathcal{A}\colon\mathbb{R}^M/\{\pm1\}\rightarrow\mathbb{R}^N$ defined by $(\mathcal{A}(x))(n):=|\langle x,\varphi_n\rangle|^2$. 
Then $\mathcal{A}$ is almost injective if and only if almost every $x\in\mathbb{R}^M$ is not in the Minkowski sum
$\operatorname{span}(\Phi_{S})^\perp\setminus\{0\}+\operatorname{span}(\Phi_{S^\mathrm{c}})^\perp\setminus\{0\}$
for all $S\subseteq\{1,\ldots,N\}$. 
More precisely, $\mathcal{A}^{-1}(\mathcal{A}(x))=\{\pm x\}$ if and only if $x\notin\operatorname{span}(\Phi_{S})^\perp\setminus\{0\}+\operatorname{span}(\Phi_{S^\mathrm{c}})^\perp\setminus\{0\}$ for any $S\subseteq\{1,\ldots,N\}$.
\end{lemma}

\begin{proof}
By the definition of the mapping $\mathcal{A}$, for $x,y\in\mathbb{R}^M$ we have $\mathcal{A}(x)=\mathcal{A}(y)$ if and only if $|\langle x,\varphi_n\rangle|=|\langle y,\varphi_n\rangle|$ for all $n\in\{1,\ldots,N\}$.
This occurs precisely when there is a subset $S\subseteq\{1,\ldots,N\}$ such that $\langle x,\varphi_n\rangle=-\langle y,\varphi_n\rangle$ for every $n\in S$ and $\langle x,\varphi_n\rangle=\langle y,\varphi_n\rangle$ for every $n\in S^\mathrm{c}$.
Thus, $\mathcal{A}^{-1}(\mathcal{A}(x))=\{\pm x\}$ if and only if for every $y\neq\pm x$ and for every $S\subseteq\{1,\ldots,N\}$, either there exists an $n\in S$ such that $\langle x+y,\varphi_n\rangle\neq0$ or an $n\in S^\mathrm{c}$ such that $\langle x-y,\varphi_n\rangle\neq0$.
We claim that this occurs if and only if $x$ is not in the Minkowski sum
$\operatorname{span}(\Phi_{S})^\perp\setminus\{0\}+\operatorname{span}(\Phi_{S^\mathrm{c}})^\perp\setminus\{0\}$
for all $S\subseteq\{1,\ldots,N\}$, which would complete the proof.
We verify the claim by seeking the contrapositive in each direction.

$(\Rightarrow)$ Suppose
$x\in\operatorname{span}(\Phi_{S})^\perp\setminus\{0\}+\operatorname{span}(\Phi_{S^\mathrm{c}})^\perp\setminus\{0\}$.
Then there exists $u\in\operatorname{span}(\Phi_{S})^\perp\setminus\{0\}$ and $v\in\operatorname{span}(\Phi_{S^\mathrm{c}})^\perp\setminus\{0\}$ such that $x=u+v$.
Taking $y:=u-v$, we see that
$x+y=2u\in\operatorname{span}(\Phi_{S})^\perp\setminus\{0\}$
and
$x-y=2v\in\operatorname{span}(\Phi_{S^\mathrm{c}})^\perp\setminus\{0\}$,
which means that for every $S\subseteq\{1,\ldots,N\}$ there is no $n\in S$ such that $\langle x+y,\varphi_n\rangle\neq0$ nor $n\in S^\mathrm{c}$ such that $\langle x-y,\varphi_n\rangle\neq0$.
Furthermore, $u$ and $v$ are nonzero, and so $y\neq\pm x$.

$(\Leftarrow)$ Suppose $y\neq\pm x$ and for every $S\subseteq\{1,\ldots,N\}$ there is no $n\in S$ such that $\langle x+y,\varphi_n\rangle\neq0$ nor $n\in S^\mathrm{c}$ such that $\langle x-y,\varphi_n\rangle\neq0$.
Then $x+y\in\operatorname{span}(\Phi_{S})^\perp\setminus\{0\}$ and $x-y\in\operatorname{span}(\Phi_{S^\mathrm{c}})^\perp\setminus\{0\}$.
Since
$x=\frac{1}{2}(x+y)+\frac{1}{2}(x-y)$,
we have that $x\in\operatorname{span}(\Phi_{S})^\perp\setminus\{0\}+\operatorname{span}(\Phi_{S^\mathrm{c}})^\perp\setminus\{0\}$.
\end{proof}

\begin{theorem}
\label{thm.almost injective and Minkowski sum proper subspace}
Consider $\Phi=\{\varphi_n\}_{n=1}^N\subseteq\mathbb{R}^M$ and the intensity measurement mapping $\mathcal{A}\colon\mathbb{R}^M/\{\pm1\}\rightarrow\mathbb{R}^N$ defined by $(\mathcal{A}(x))(n):=|\langle x,\varphi_n\rangle|^2$. 
Suppose $\Phi$ spans $\mathbb{R}^M$ and each $\varphi_n$ is nonzero.
Then $\mathcal{A}$ is almost injective if and only if the Minkowski sum
$\operatorname{span}(\Phi_{S})^\perp+\operatorname{span}(\Phi_{S^\mathrm{c}})^\perp$
is a proper subspace of $\mathbb{R}^M$ for each nonempty proper subset $S\subseteq\{1,\ldots,N\}$.
\end{theorem}

Note that the above result is not terribly surprising considering Lemma~\ref{lem.almost injective and Minkowski sum}, as the new condition involves a simpler Minkowski sum in exchange for additional (reasonable and testable) assumptions on $\Phi$.
The proof of this theorem amounts to measuring the difference between the two Minkowski sums:

\begin{proof}[Proof of Theorem~\ref{thm.almost injective and Minkowski sum proper subspace}]
We start with the following claim:
\begin{equation}
\label{eq.Minkowski sum equality}
\operatorname{span}(\Phi_{S})^\perp\setminus\{0\}+\operatorname{span}(\Phi_{S^\mathrm{c}})^\perp\setminus\{0\}
=\left(\operatorname{span}(\Phi_{S})^\perp+\operatorname{span}(\Phi_{S^\mathrm{c}})^\perp\right)
\setminus\left(\operatorname{span}(\Phi_{S})^\perp\cup\operatorname{span}(\Phi_{S^\mathrm{c}})^\perp\right).
\end{equation}
Before verifying this claim, let's first use it to prove the theorem.
From Lemma \ref{lem.almost injective and Minkowski sum} we know that $\mathcal{A}$ is almost injective if and only if almost every $x\in\mathbb{R}^M$ is not in the Minkowski sum
$\operatorname{span}(\Phi_{S})^\perp\setminus\{0\}+\operatorname{span}(\Phi_{S^\mathrm{c}})^\perp\setminus\{0\}$
for any $S\subseteq\{1,\ldots,N\}$.
In other words, the Lebesgue measure of this Minkowski sum is zero for each $S\subseteq\{1,\ldots,N\}$.
By \eqref{eq.Minkowski sum equality}, this equivalently means that the Lebesgue measure of
$\left(\operatorname{span}(\Phi_{S})^\perp+\operatorname{span}(\Phi_{S^\mathrm{c}})^\perp\right)
\setminus\left(\operatorname{span}(\Phi_{S})^\perp\cup\operatorname{span}(\Phi_{S^\mathrm{c}})^\perp\right)$
is zero for each $S\subseteq\{1,\ldots,N\}$.
Since $\Phi$ spans $\mathbb{R}^M$, this set is empty (and therefore has Lebesgue measure zero) when $S=\emptyset$ or $S=\{1,\ldots,N\}$.
Also, since each $\varphi_n$ is nonzero, we know that $\operatorname{span}(\Phi_{S})^\perp$ and $\operatorname{span}(\Phi_{S^\mathrm{c}})^\perp$ are proper subspaces of $\mathbb{R}^M$ whenever $S$ is a nonempty proper subset of $\{1,\ldots,N\}$, and so in these cases both subspaces must have Lebesgue measure zero.
As such, we have that for every nonempty proper subset $S\subseteq\{1,\ldots,N\}$,
\begin{align*}
&\operatorname{Leb}\left[\left(\operatorname{span}(\Phi_{S})^\perp+\operatorname{span}(\Phi_{S^\mathrm{c}})^\perp\right)
\setminus\left(\operatorname{span}(\Phi_{S})^\perp\cup\operatorname{span}(\Phi_{S^\mathrm{c}})^\perp\right)\right]\\
&\qquad\qquad\geq\operatorname{Leb}\left(\operatorname{span}(\Phi_{S})^\perp+\operatorname{span}(\Phi_{S^\mathrm{c}})^\perp\right)
-\operatorname{Leb}\left(\operatorname{span}(\Phi_{S})^\perp\right)-\operatorname{Leb}\left(\operatorname{span}(\Phi_{S^\mathrm{c}})^\perp\right)\\
&\qquad\qquad=\operatorname{Leb}\left(\operatorname{span}(\Phi_{S})^\perp+\operatorname{span}(\Phi_{S^\mathrm{c}})^\perp\right)\\
&\qquad\qquad\geq\operatorname{Leb}\left[\left(\operatorname{span}(\Phi_{S})^\perp+\operatorname{span}(\Phi_{S^\mathrm{c}})^\perp\right)\setminus\left(\operatorname{span}(\Phi_{S})^\perp\cup\operatorname{span}(\Phi_{S^\mathrm{c}})^\perp\right)\right].
\end{align*}
In summary, $\left(\operatorname{span}(\Phi_{S})^\perp+\operatorname{span}(\Phi_{S^\mathrm{c}})^\perp\right)
\setminus\left(\operatorname{span}(\Phi_{S})^\perp\cup\operatorname{span}(\Phi_{S^\mathrm{c}})^\perp\right)$
having Lebesgue measure zero for each $S\subseteq\{1,\ldots,N\}$ is equivalent to $\operatorname{span}(\Phi_{S})^\perp+\operatorname{span}(\Phi_{S^\mathrm{c}})^\perp$ having Lebesgue measure zero for each nonempty proper subset $S\subseteq\{1,\ldots,N\}$, which in turn is equivalent to the Minkowski sum $\operatorname{span}(\Phi_{S})^\perp+\operatorname{span}(\Phi_{S^\mathrm{c}})^\perp$ being a proper subspace of $\mathbb{R}^M$ for each nonempty proper subset $S\subseteq\{1,\ldots,N\}$, as desired.

Thus, to complete the proof we must verify the claim~\eqref{eq.Minkowski sum equality}.
We will do so by verifying both inclusions.
Clearly
$\operatorname{span}(\Phi_{S})^\perp\setminus\{0\}+\operatorname{span}(\Phi_{S^\mathrm{c}})^\perp\setminus\{0\}$
is a subset of
$\operatorname{span}(\Phi_{S})^\perp+\operatorname{span}(\Phi_{S^\mathrm{c}})^\perp$,
so to prove $\subseteq$ in \eqref{eq.Minkowski sum equality}, it suffices to show that
\begin{equation}
\label{eq.Minkowski sum empty intersection}
\left(\operatorname{span}(\Phi_{S})^\perp\setminus\{0\}+\operatorname{span}(\Phi_{S^\mathrm{c}})^\perp\setminus\{0\}\right)
\cap\left(\operatorname{span}(\Phi_{S})^\perp\cup\operatorname{span}(\Phi_{S^\mathrm{c}})^\perp\right)
=\emptyset.
\end{equation}
Assuming to the contrary, then without loss of generality there exist elements
$a\in\operatorname{span}(\Phi_{S})^\perp$,
$b\in\operatorname{span}(\Phi_{S})^\perp\setminus\{0\}$,
and $c\in\operatorname{span}(\Phi_{S^\mathrm{c}})^\perp\setminus\{0\}$
such that $a=b+c$.
But this means that $a-b=c\neq0$ is in both $\operatorname{span}(\Phi_{S})^\perp$ and $\operatorname{span}(\Phi_{S^\mathrm{c}})^\perp$,
contradicting the assumption that the vectors $\Phi=\{\varphi_n\}_{n=1}^N$ span $\mathbb{R}^M$.
To prove $\supseteq$ in \eqref{eq.Minkowski sum equality}, note that \eqref{eq.Minkowski sum empty intersection} tells us it is equivalent to show the containment
\begin{equation*}
\operatorname{span}(\Phi_{S})^\perp+\operatorname{span}(\Phi_{S^\mathrm{c}})^\perp
\subseteq\left(\operatorname{span}(\Phi_{S})^\perp\setminus\{0\}+\operatorname{span}(\Phi_{S^\mathrm{c}})^\perp\setminus\{0\}\right)
\cup\left(\operatorname{span}(\Phi_{S})^\perp\cup\operatorname{span}(\Phi_{S^\mathrm{c}})^\perp\right).
\end{equation*}
To this end, let
$a\in\operatorname{span}(\Phi_{S})^\perp$
and $b\in\operatorname{span}(\Phi_{S^\mathrm{c}})^\perp$
so that
$a+b\in\operatorname{span}(\Phi_{S})^\perp+\operatorname{span}(\Phi_{S^\mathrm{c}})^\perp$. 
Then the inclusion follows from observing the following cases:

\begin{itemize}

\item[(I)] 
Suppose $a$ and $b$ are nonzero.
Then $a\in\operatorname{span}(\Phi_{S})^\perp\setminus\{0\}$
and $b\in\operatorname{span}(\Phi_{S^\mathrm{c}})^\perp\setminus\{0\}$,
implying that
$a+b\in\operatorname{span}(\Phi_{S})^\perp\setminus\{0\}+\operatorname{span}(\Phi_{S^\mathrm{c}})^\perp\setminus\{0\}$.

\item[(II)]
Suppose exactly one of $a$ and $b$ are nonzero (without loss of generality that $a\neq0$ and $b=0$).
Then $a+b=a\in\operatorname{span}(\Phi_{S})^\perp$,
implying that
$a+b\in\operatorname{span}(\Phi_{S})^\perp\cup\operatorname{span}(\Phi_{S^\mathrm{c}})^\perp$.

\item[(III)]
Suppose $a$ and $b$ are both zero.
Then
$a+b\in\operatorname{span}(\Phi_{S})^\perp\cup\operatorname{span}(\Phi_{S^\mathrm{c}})^\perp$.
\end{itemize}
Having confirmed both inclusions of our initial claim~\eqref{eq.Minkowski sum equality}, the proof is complete.
\end{proof}

At this point, consider the following stronger restatement of Theorem~\ref{thm.almost injective and Minkowski sum proper subspace}: 
``Suppose each $\varphi_n$ is nonzero.
Then $\mathcal{A}$ is almost injective if and only if $\Phi$ spans $\mathbb{R}^M$ and the Minkowski sum
$\operatorname{span}(\Phi_{S})^\perp+\operatorname{span}(\Phi_{S^\mathrm{c}})^\perp$
is a proper subspace of $\mathbb{R}^M$ for each nonempty proper subset $S\subseteq\{1,\ldots,N\}$.''
Note that we can move the spanning assumption into the condition because if $\Phi$ does not span, then we can decompose almost every $x\in\mathbb{R}^M$ as $x=u+v$ such that $u\in\operatorname{span}(\Phi)$ and $v\in\operatorname{span}(\Phi)^\perp$ with $v\neq0$, and defining $y:=u-v$ then gives $\mathcal{A}(y)=\mathcal{A}(x)$ despite the fact that $y\neq\pm x$.
As for the assumption that the $\varphi_n$'s are nonzero, we note that having $\varphi_n=0$ amounts to having the $n$th entry of $\mathcal{A}(x)$ be zero for all $x$.
As such, $\Phi$ yields almost injectivity precisely when the nonzero members of $\Phi$ together yield almost injectivity.
With this identification, the stronger restatement of Theorem~\ref{thm.almost injective and Minkowski sum proper subspace} above can be viewed as a complete characterization of almost injectivity.
Next, we will replace the Minkowski sum condition with a rather elegant condition involving the ranks of $\Phi_S$ and $\Phi_{S^\mathrm{c}}$ by applying the following lemma:

\begin{lemma}[Inclusion-exclusion principle for subspaces]
\label{lem.subspace inclusion exclusion}
Let $U$ and $V$ be subspaces of a common vector space.
Then $\dim(U+V)=\dim U+\dim V-\dim(U\cap V)$.
\end{lemma}

\begin{proof}
Let $A$ be a basis for $U\cap V$ and let $B$ and $C$ be bases for $U$ and $V$, respectively, such that $A\subseteq B$ and $A\subseteq C$.
It can be shown that $A\cup B\cup C$ forms a basis for $U+V$, which implies that
\begin{equation*}
\dim(U+V)
=|A|+|B\setminus A|+|C\setminus A|
=|B|+|C|-|A|
=\dim U+\dim V-\dim(U\cap V).
\qedhere
\end{equation*}
\end{proof}

\begin{theorem}
\label{thm.almost injective and sum rank >M}
Consider $\Phi=\{\varphi_n\}_{n=1}^N\subseteq\mathbb{R}^M$ and the intensity measurement mapping $\mathcal{A}\colon\mathbb{R}^M/\{\pm1\}\rightarrow\mathbb{R}^N$ defined by $(\mathcal{A}(x))(n):=|\langle x,\varphi_n\rangle|^2$. 
Suppose each $\varphi_n$ is nonzero.
Then $\mathcal{A}$ is almost injective if and only if $\Phi$ spans $\mathbb{R}^M$ and $\operatorname{rank}\Phi_{S}+\operatorname{rank}\Phi_{S^\mathrm{c}}>M$
for each nonempty proper subset $S\subseteq\{1,\ldots,N\}$.
\end{theorem}

\begin{proof}
Considering the discussion after the proof of Theorem~\ref{thm.almost injective and Minkowski sum proper subspace}, it suffices to assume that $\Phi$ spans $\mathbb{R}^M$.
Furthermore, considering Theorem~\ref{thm.almost injective and Minkowski sum proper subspace}, it suffices to characterize when $\dim\left(\operatorname{span}(\Phi_{S})^\perp+\operatorname{span}(\Phi_{S^\mathrm{c}})^\perp\right)<M$.
By Lemma \ref{lem.subspace inclusion exclusion}, we have
\begin{align*}
&\dim\left(\operatorname{span}(\Phi_{S})^\perp+\operatorname{span}(\Phi_{S^\mathrm{c}})^\perp\right)\\
&\qquad\qquad=\dim\left(\operatorname{span}(\Phi_{S})^\perp\right)
+\dim\left(\operatorname{span}(\Phi_{S^\mathrm{c}})^\perp\right)
-\dim\left(\operatorname{span}(\Phi_{S})^\perp\cap\operatorname{span}(\Phi_{S^\mathrm{c}})^\perp\right).
\end{align*}
Since $\Phi$ is assumed to span $\mathbb{R}^M$, we also have that
$\operatorname{span}(\Phi_{S})^\perp\cap\operatorname{span}(\Phi_{S^\mathrm{c}})^\perp=\{0\}$,
and so
\begin{align*}
\dim\left(\operatorname{span}(\Phi_{S})^\perp+\operatorname{span}(\Phi_{S^\mathrm{c}})^\perp\right)
&=\Big(M-\dim\left(\operatorname{span}(\Phi_{S})\right)\Big)
+\Big(M-\dim\left(\operatorname{span}(\Phi_{S^\mathrm{c}})\right)\Big)
-0\\
&=2M-\operatorname{rank}\Phi_{S}-\operatorname{rank}\Phi_{S^\mathrm{c}}.
\end{align*}
As such, $\dim\left(\operatorname{span}(\Phi_{S})^\perp+\operatorname{span}(\Phi_{S^\mathrm{c}})^\perp\right)<M$ precisely when $\operatorname{rank}\Phi_{S}+\operatorname{rank}\Phi_{S^\mathrm{c}}>M$.
\end{proof}

At this point, we point out some interesting consequences of Theorem~\ref{thm.almost injective and sum rank >M}.
First of all, $\Phi$ cannot be almost injective if $N<M+1$ since $\operatorname{rank}\Phi_{S}+\operatorname{rank}\Phi_{S^\mathrm{c}}\leq|S|+|S^\mathrm{c}|=N$.
Also, in the case where $N=M+1$, we note that $\Phi$ is almost injective precisely when $\Phi$ is \textit{full spark}, that is, every size-$M$ subcollection is a spanning set (note this implies that all of the $\varphi_n$'s are nonzero).
In fact, every full spark $\Phi$ with $N\geq M+1$ yields almost injective intensity measurements, which in turn implies that a generic $\Phi$ yields almost injectivity when $N\geq M+1$~\cite{BalanCE:06}.
This is in direct analogy with injectivity in the real case; here, injectivity requires $N\geq 2M-1$, injectivity with $N=2M-1$ is equivalent to being full spark, and being full spark suffices for injectivity whenever $N\geq 2M-1$~\cite{BalanCE:06}.
Another thing to check is that the condition for injectivity implies the condition for almost injectivity (it does).

Having established that full spark ensembles of size $N\geq M+1$ yield almost injective intensity measurements, we note that checking whether a matrix is full spark is $\NP$-hard in general~\cite{Khachiyan:95}.
Granted, there are a few explicit constructions of full spark ensembles which can be used~\cite{AlexeevCM:12,PuschelK:05}, but it would be nice to have a condition which is not computationally difficult to test in general.
We provide one such condition in the following theorem, but first, we briefly review the requisite frame theory.

A \textit{frame} is an ensemble $\Phi=\{\varphi_n\}_{n=1}^N\subseteq\mathbb{R}^M$ together with \textit{frame bounds} $0<A\leq B<\infty$ with the property that for every $x\in\mathbb{R}^M$,
\begin{equation*}
A\|x\|^2
\leq\sum_{n=1}^N|\langle x,\varphi_n\rangle|^2
\leq B\|x\|^2.
\end{equation*}
When $A=B$, the frame is said to be \textit{tight}, and such frames come with a painless reconstruction formula:
\begin{equation*}
x=\frac{1}{A}\sum_{n=1}^N\langle x,\varphi_n\rangle\varphi_n.
\end{equation*}
To be clear, the theory of frames originated in the context of infinite-dimensional Hilbert spaces~\cite{DaubechiesGM:86,DuffinS:52}, and frames have since been studied in finite-dimensional settings, primarily because this is the setting in which they are applied computationally.
Of particular interest are so-called \textit{unit norm tight frames (UNTFs)}, which are tight frames whose frame elements have unit norm: $\|\varphi_n\|=1$ for every $n=1,\ldots,N$.
Such frames are useful in applications; for example, if you encode a signal $x$ using frame coefficients $\langle x,\varphi_n\rangle$ and transmit these coefficients across a channel, then UNTFs are optimally robust to noise~\cite{GoyalVT:98} and one erasure~\cite{CasazzaK:03}.
Intuitively, this optimality comes from the fact that frame elements of a UNTF are particularly well-distributed in the unit sphere~\cite{BenedettoF:03}.
Another pleasant feature of UNTFs is that it is straightforward to test whether a given frame is a UNTF:
Letting $\Phi=[\varphi_1\cdots\varphi_N]$ denote an $M\times N$ matrix whose columns are the frame elements, then $\Phi$ is a UNTF precisely when each of the following occurs simultaneously:
\begin{itemize}
\item[(i)] the rows have equal norm
\item[(ii)] the rows are orthogonal
\item[(iii)] the columns have unit norm
\end{itemize}
(This is a direct consequence of the tight frame's reconstruction formula and the fact that a UNTF has unit-norm frame elements; furthermore, since the columns have unit norm, it is not difficult to see that the rows will necessarily have norm $\sqrt{N/M}$.)
In addition to being able to test that an ensemble is a UNTF, various UNTFs can be constructed using \textit{spectral tetris}~\cite{CasazzaFMWZ:11} (though such frames necessarily have $N\geq 2M$), and \textit{every} UNTF can be constructed using the recent theory of \textit{eigensteps}~\cite{CahillFMPS:13,FickusMPS:13}.
Now that UNTFs have been properly introduced, we relate them to almost injectivity for phase retrieval:

\begin{theorem}
\label{thm.almost injectivity and relatively prime UNTFs}
If $M$ and $N$ are relatively prime, then every unit norm tight frame $\Phi=\{\varphi_n\}_{n=1}^N\subseteq\mathbb{R}^M$ yields almost injective intensity measurements.
\end{theorem}

\begin{proof}
Pick a nonempty proper subset $S\subseteq\{1,\ldots,N\}$.
By Theorem~\ref{thm.almost injective and sum rank >M}, it suffices to show that
$\operatorname{rank}\Phi_{S}+\operatorname{rank}\Phi_{S^\mathrm{c}}>M$,
or equivalently,
$\operatorname{rank}\Phi_{S}\Phi_{S}^*+\operatorname{rank}\Phi_{S^\mathrm{c}}\Phi_{S^{\mathrm{c}}}^*>M$.
Note that since $\Phi$ is a unit norm tight frame, we also have
\begin{equation*}
\Phi_{S}\Phi_{S}^*+\Phi_{S^\mathrm{c}}\Phi_{S^{\mathrm{c}}}^*=\Phi\Phi^*=\tfrac{N}{M}I,
\end{equation*}
and so $\Phi_{S}\Phi_{S}^*$ and $\Phi_{S^\mathrm{c}}\Phi_{S^{\mathrm{c}}}^*$ are simultaneously diagonalizable, i.e., there exists a unitary matrix $U$ and diagonal matrices $D_1$ and $D_2$ such that
\begin{align*}
UD_1U^*+UD_2U^*=\Phi_{S}\Phi_{S}^*+\Phi_{S^\mathrm{c}}\Phi_{S^{\mathrm{c}}}^*=\tfrac{N}{M}I.
\end{align*}
Conjugating by $U^*$, this then implies that $D_1+D_2=\tfrac{N}{M}I$.
Let $L_1\subseteq\{1,\ldots,M\}$ denote the diagonal locations of the nonzero entries in $D_1$, and $L_2\subseteq\{1,\ldots,M\}$ similarly for $D_2$.
To complete the proof, we need to show that $|L_1|+|L_2|>M$ (since $|L_1|+|L_2|=\operatorname{rank}\Phi_{S}\Phi_{S}^*+\operatorname{rank}\Phi_{S^\mathrm{c}}\Phi_{S^{\mathrm{c}}}^*$).
Note that $L_1\cup L_2\neq\{1,\ldots,M\}$ would imply that $D_1+D_2$ has at least one zero in its diagonal, contradicting the fact that $D_1+D_2$ is a nonzero multiple of the identity; as such, $L_1\cup L_2=\{1,\ldots,M\}$ and $|L_1|+|L_2|\geq M$.
We claim that this inequality is strict due to the assumption that $M$ and $N$ are relatively prime.
To see this, it suffices to show that $L_1\cap L_2$ is nonempty.
Suppose to the contrary that $L_1$ and $L_2$ are disjoint.
Then since $D_1+D_2=\tfrac{N}{M}I$, every nonzero entry in $D_1$ must be $N/M$.
Since $S$ is a nonempty proper subset of $\{1,\ldots,N\}$, this means that there exists $K\in(0,M)$ such that $D_1$ has $K$ entries which are $N/M$ and $M-K$ which are $0$.
Thus,
\begin{equation*}
|S|
=\operatorname{Tr}[\Phi_{S}^*\Phi_{S}]
=\operatorname{Tr}[\Phi_{S}\Phi_{S}^*]
=\operatorname{Tr}[UD_1U^*]
=\operatorname{Tr}[D_1]
=K(N/M),
\end{equation*}
implying that $N/M=|S|/K$ with $K\neq M$ and $|S|\neq N$.
Since this contradicts the assumption that $N/M$ is in lowest form, we have the desired result.
\end{proof}

In general, whether a UNTF $\Phi$ yields almost injective intensity measurements is determined by whether it is \textit{orthogonally partitionable}: $\Phi$ is orthogonally partitionable if there exists a partition $S\sqcup S^\mathrm{c}=\{1,\ldots,N\}$ such that $\operatorname{span}(\Phi_S)$ is orthogonal to $\operatorname{span}(\Phi_{S^\mathrm{c}})$.
Specifically, a UNTF yields almost injective intensity measurements precisely when it is not orthogonally partitionable.
Historically, this property of UNTFs has been pivotal to the understanding of singularities in the algebraic variety of UNTFs~\cite{DykemaS:06}, and it has also played a key role in solutions to the Paulsen problem~\cite{BodmannC:10,CasazzaFM:12}.
However, it is not clear in general how to efficiently test for this property; this is why Theorem~\ref{thm.almost injectivity and relatively prime UNTFs} focuses on such a special case.

\begin{figure}[t]
\begin{center}
\includegraphics[width=0.4\textwidth]{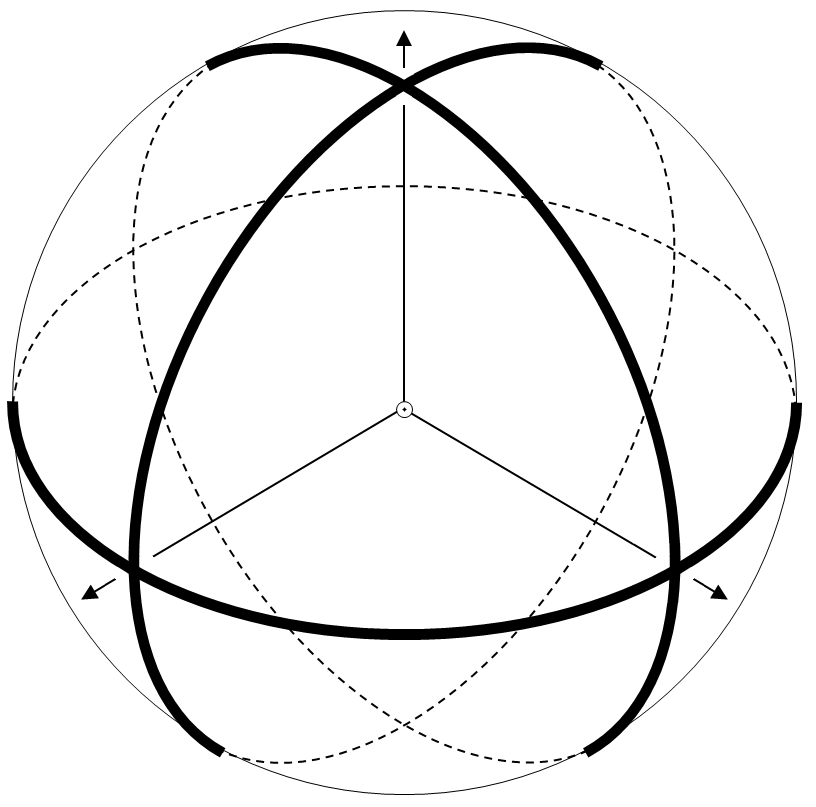}
\caption{
\label{fig.simplex}
The simplex in $\mathbb{R}^3$.
Pointing out of the page is the vector $\smash{\frac{1}{\sqrt{3}}(1,1,1)}$, while the other vectors are the three permutations of $\smash{\frac{1}{\sqrt{3}}(1,-1,-1)}$.
Together, these four vectors form a unit norm tight frame, and since $M=3$ and $N=4$ are relatively prime, these yield almost injective intensity measurements in accordance with Theorem~\ref{thm.almost injectivity and relatively prime UNTFs}.
For this ensemble, the points $x$ such that $\mathcal{A}^{-1}(\mathcal{A}(x))\neq\{\pm x\}$ are contained in the three coordinate planes.
Above, we depict the intersection between these planes and the unit sphere.
According to Theorem~\ref{thm.full spark Mx(M+1) solution is NP-complete}, performing phase retrieval with simplices such as this is $\NP$-hard.
}
\end{center}
\end{figure}

\section{The computational complexity of phase retrieval}

The previous section characterized the real ensembles which yield almost injective intensity measurements.
The benefit of seeking almost injectivity instead of injectivity is that we can get away with much smaller ensembles.
For example, a full spark ensemble in $\mathbb{R}^M$ of size $M+1$ suffices for almost injectivity, while $2M-1$ measurements are required for injectivity.
In this section, we demonstrate that this savings in the number of measurements can come at a substantial price in computational requirements for phase retrieval.
In particular, we consider the following problem:

\begin{problem}
\label{prob.full spark Mx(M+1) solution}
Let $\mathcal{F}=\{\Phi_M\}_{M=2}^\infty$ be a family of ensembles $\Phi_M=\{\varphi_{M;n}\}_{n=1}^{N(M)}\subseteq\mathbb{R}^M$, where $N(M)=\poly(M)$.
Then $\textsc{ConsistentIntensities}[\mathcal{F}]$ is the following problem: 
Given $M\geq 2$ and a rational sequence $\{b_n\}_{n=1}^{N(M)}$, does there exist $x\in\mathbb{R}^M$ such that $|\langle x,\varphi_{M;n}\rangle|=b_n$ for every $n=1,\ldots,N(M)$?
\end{problem}

In this section, we will evaluate the computational complexity of $\textsc{ConsistentIntensities}[\mathcal{F}]$ for a large class of families of small ensembles $\mathcal{F}$, but first, we briefly review the main concepts involved.
Complexity theory is chiefly concerned with \textit{complexity classes}, which are sets of problems that share certain computational requirements, such as time or space.
For example, the complexity class $\P$ is the set of problems which can be solved in an amount of time that is bounded by some polynomial of the bit-length of the input.
As another example, $\NP$ contains all problems for which an affirmative answer comes with a certificate that can be verified in polynomial time; note that $\P\subseteq\NP$ since for every problem $A\in\P$, one may ignore the certificate and find the affirmative answer in polynomial time.
One key tool that is used to evaluate the complexity of a problem is called \textit{polynomial-time reduction}.
This is a polynomial-time algorithm that solves a problem $A$ by exploiting an oracle which solves another problem $B$, indicating that solving $A$ is no harder than solving $B$ (up to polynomial factors in time); if such a reduction exists, we write $A\leq B$.
For example, any efficient phase retrieval procedure for $\mathcal{F}$ can be used as a subroutine to solve $\textsc{ConsistentIntensities}[\mathcal{F}]$, indicating that phase retrieval for $\mathcal{F}$ is at least as hard as $\textsc{ConsistentIntensities}[\mathcal{F}]$.
A problem $B$ is called \textit{$\NP$-hard} if $B\geq A$ for every problem $A\in\NP$.
Note that since $\leq$ is transitive, it suffices to show that $B\geq C$ for some $\NP$-hard problem $C$.
Finally, a problem $B$ is called \textit{$\NP$-complete} if $B\in\NP$ is $\NP$-hard; intuitively, $\NP$-complete problems are the hardest of problems in $\NP$.
It is an open problem whether $\P=\NP$, but inequality is widely believed~\cite{Cook:13}; note that under this assumption, $\NP$-hard problems have no computationally efficient solution.
This provides a proper context for the main result of this section:

\begin{theorem}
\label{thm.full spark Mx(M+1) solution is NP-complete}
Let $\mathcal{F}=\{\Phi_M\}_{M=2}^\infty$ be a family of full spark ensembles $\Phi_M=\{\varphi_{M;n}\}_{n=1}^{M+1}\subseteq\mathbb{R}^M$ with rational entries that can be computed in polynomial time.
Then $\textsc{ConsistentIntensities}[\mathcal{F}]$ is $\NP$-complete.
\end{theorem}

Note that since the ensembles $\Phi_M$ are full spark, the existence of a solution to the phase retrieval problem $|\langle x,\varphi_{M;n}\rangle|=b_n$ for every $n=1,\ldots,M+1$ implies uniqueness by Theorem~\ref{thm.almost injective and sum rank >M}.
Before proving this theorem, we first relate it to a previous hardness result from~\cite{SahinoglouC:91}.
Specifically, this result can be restated using the terminology in this paper as follows:
There exists a family $\mathcal{F}=\{\Phi_M\}_{M=2}^\infty$ of ensembles $\Phi_M=\{\varphi_{M;n}\}_{n=1}^{2M}\subseteq\mathbb{C}^M$, each of which yielding almost injective intensity measurements, such that $\textsc{ConsistentIntensities}[\mathcal{F}]$ is $\NP$-complete.
Interestingly, these are the smallest possible almost injective ensembles in the complex case, and we suspect that the result can be strengthened to the obvious analogy of Theorem~\ref{thm.full spark Mx(M+1) solution is NP-complete}:

\begin{conjecture}
Let $\mathcal{F}=\{\Phi_M\}_{M=2}^\infty$ be a family of ensembles $\Phi_M=\{\varphi_{M;n}\}_{n=1}^{2M}\subseteq\mathbb{C}^M$ which yield almost injective intensity measurements and have complex rational entries that can be computed in polynomial time.
Then $\textsc{ConsistentIntensities}[\mathcal{F}]$ is $\NP$-complete.
\end{conjecture}

To prove Theorem~\ref{thm.full spark Mx(M+1) solution is NP-complete}, we devise a polynomial-time reduction from the following problem which is well-known to be $\NP$-complete~\cite{Karp:72}:

\begin{problem}[\textsc{SubsetSum}]
\label{prob.subset sum}
Given a finite collection of integers $A$ and an integer $z$, does there exist a subset $S\subseteq A$ such that $\sum_{a\in S}a=z$?
\end{problem}

\begin{proof}[Proof of Theorem~\ref{thm.full spark Mx(M+1) solution is NP-complete}]
We first show that $\textsc{ConsistentIntensities}[\mathcal{F}]$ is in $\NP$.
Note that if there exists an $x\in\mathbb{R}^M$ such that $|\langle x,\varphi_{M;n}\rangle|=b_n$ for every $n=1,\ldots,M+1$, then $x$ will have all rational entries.
Indeed, $v:=\Phi_M^*x$ has all rational entries, being a signed version of $\{b_n\}_{n=1}^{M+1}$, and so $x=(\Phi_M\Phi_M^*)^{-1}\Phi_Mv$ is also rational.
Thus, we can view $x$ as a certificate of finite bit-length, and for each $n=1,\ldots,M+1$, we know that $|\langle x,\varphi_{M;n}\rangle|=b_n$ can be verified in time which is polynomial in this bit-length, as desired.

Now we show that $\textsc{ConsistentIntensities}[\mathcal{F}]$ is $\NP$-hard by reduction from \textsc{SubsetSum}.
To this end, take a finite collection of integers $A$ and an integer $z$.
Set $M:=|A|$ and label the members of $A$ as $\{a_m\}_{m=1}^M$.
Let $\Psi$ denote the $M\times M$ matrix whose columns are the first $M$ members of $\Phi_M$.
Since $\Phi_M$ is full spark, $\Psi$ is invertible and $\Psi^{-1}\Phi_M$ has the form $[I~w]$, where $w$ has all nonzero entries; indeed, if the $m$th entry of $w$ were zero, then $\Phi_M\setminus\{\varphi_{M;m}\}$ would not span, violating full spark.
Now define 
\begin{equation}
\label{eq.defn of b}
b_n:=
\left\{
\begin{array}{ll}
\displaystyle{\bigg|\frac{a_n}{w_n}\bigg|}&\mbox{if }n=1,\ldots,M\\
\displaystyle{\bigg|2z-\sum_{m=1}^Ma_m\bigg|}&\mbox{if }n=M+1.
\end{array}
\right.
\end{equation}
We claim that an oracle for $\textsc{ConsistentIntensities}[\mathcal{F}]$ would return ``yes'' from the inputs $M$ and $\{b_n\}_{n=1}^{M+1}$ defined above if and only if there exists a subset $S\subseteq A$ such that $\sum_{a\in S}a=z$, which would complete the reduction.

To prove our claim, we start with ($\Rightarrow$):
Suppose there exists $x\in\mathbb{R}^M$ such that $|\langle x,\varphi_{M;n}\rangle|=b_n$ for every $n=1,\ldots,M+1$.
Then $y:=\Psi^*x$ satisfies $|\langle y,\Psi^{-1}\varphi_{M;n}\rangle|=b_n$ for every $n=1,\ldots,M+1$.
Since $\Psi^{-1}\Phi_M=[I~w]$, then by \eqref{eq.defn of b}, the entries of $y$ satisfy 
\begin{equation*}
|y_m|=\left|\frac{a_m}{w_m}\right|\quad\forall m=1,\ldots,M,
\qquad\qquad
\bigg|\sum_{m=1}^My_mw_m\bigg|=\bigg|2z-\sum_{m=1}^Ma_m\bigg|.
\end{equation*}
By the first equation above, there exists a sequence $\{\varepsilon_m\}_{m=1}^M$ of $\pm1$'s such that $y_m=\varepsilon_ma_m/w_m$ for every $m=1,\ldots,M$, and so the second equation above gives
\begin{equation*}
\bigg|2z-\sum_{m=1}^Ma_m\bigg|
=\bigg|\sum_{m=1}^My_mw_m\bigg|
=\bigg|\sum_{m=1}^M\varepsilon_ma_m\bigg|
=\bigg|\sum_{\substack{m=1\\\varepsilon_m=1}}^Ma_m-\sum_{\substack{m=1\\\varepsilon_m=-1}}^Ma_m\bigg|
=\bigg|2\sum_{\substack{m=1\\\varepsilon_m=1}}^Ma_m-\sum_{m=1}^Ma_m\bigg|.
\end{equation*}
Removing the absolute values, this means the left-hand side above is equal to the right-hand side, up to a sign factor.
At this point, isolating $z$ reveals that $z=\sum_{m\in S}a_m$, where $S$ is either $\{m:\varepsilon_m=1\}$ or $\{m:\varepsilon_m=-1\}$, depending on the sign factor.

For ($\Leftarrow$), suppose there is a subset $S\subseteq\{1,\ldots,M\}$ such that $z=\sum_{m\in S}a_m$.
Define $\varepsilon_m:=1$ when $m\in S$ and $\varepsilon_m:=-1$ when $m\not\in S$.
Then
\begin{equation*}
\bigg|\sum_{m=1}^M\varepsilon_ma_m\bigg|
=\bigg|\sum_{\substack{m=1\\\varepsilon_m=1}}^Ma_m-\sum_{\substack{m=1\\\varepsilon_m=-1}}^Ma_m\bigg|
=\bigg|2\sum_{\substack{m=1\\\varepsilon_m=1}}^Ma_m-\sum_{m=1}^Ma_m\bigg|
=\bigg|2z-\sum_{m=1}^Ma_m\bigg|.
\end{equation*}
By the analysis from the ($\Rightarrow$) direction, taking $y_m:=\varepsilon_ma_m/w_m$ for each $m=1,\ldots,M$ then ensures that $|\langle y,\Psi^{-1}\varphi_{M;n}\rangle|=b_n$ for every $n=1,\ldots,M+1$, which in turn ensures that $x:=(\Psi^*)^{-1}y$ satisfies $|\langle x,\varphi_{M;n}\rangle|=b_n$ for every $n=1,\ldots,M+1$.
\end{proof}

Based on Theorem~\ref{thm.full spark Mx(M+1) solution is NP-complete}, there is no polynomial-time algorithm to perform phase retrieval for minimal almost injective ensembles, assuming $\P\neq\NP$.
On the other hand, there exist ensembles of size $2M-1$ for which phase retrieval is particularly efficient.
For example, letting $\delta_{M;m}\in\mathbb{R}^M$ denote the $m$th identity basis element, consider the ensemble $\Phi_M:=\{\delta_{M;m}\}_{m=1}^M\cup\{\delta_{M;1}+\delta_{M;m}\}_{m=2}^M$; then one can reconstruct (up to global phase) any $x$ whose first entry is nonzero by first taking $\hat{x}[1]:=|\langle x,\delta_{M;1}\rangle|$, and then taking
\begin{equation*}
\hat{x}[m]:=
\frac{1}{2\hat{x}[1]}\Big(|\langle x,\delta_{M;1}+\delta_{M;m}\rangle|^2-|\langle x,\delta_{M;1}\rangle|^2-|\langle x,\delta_{M;m}\rangle|^2\Big)
\qquad
\forall m=2,\ldots,M.
\end{equation*}
Intuitively, we expect a redundancy threshold that determines whether phase retrieval can be efficient, and this suggests the following open problem:
What is the smallest $C$ for which there exists a family of ensembles of size $N=CM+o(M)$ such that phase retrieval can be performed in polynomial time?

\section*{Acknowledgments}
The authors thank the Norbert Wiener Center for Harmonic Analysis and Applications at the University of Maryland, College Park for hosting a workshop on phase retrieval that helped solidify the main ideas in the almost injectivity portion of this paper. 
This work was supported by NSF DMS 1042701 and 1321779.
The views expressed in this article are those of the authors and do not reflect the official policy or position of the United States Air Force, Department of Defense, or the U.S.~Government.

\end{document}